\documentclass{amsart}
\usepackage[alphabetic]{amsrefs}

\usepackage{amsmath}
\usepackage{amssymb}
\usepackage{mathrsfs}
\usepackage{pstricks,pst-node,xcolor}

\DeclareMathOperator{\N}{\mathbf{N}}

\DeclareMathOperator{\R}{\mathbf{R}}
\DeclareMathOperator{\C}{\mathbf{C}}

\DeclareMathOperator{\sgn}{\operatorname{sgn}}

\DeclareMathOperator{\Aug}{\operatorname{Aug}}
\DeclareMathOperator{\Def}{\operatorname{Def}}

\newcommand{\beq}[1]{\begin{equation}\label{#1}}
\newcommand{\eeq}{\end{equation}}
\newcommand{\bpm}{\begin{pmatrix}}
\newcommand{\epm}{\end{pmatrix}}
\newcommand{\bsqm}{\begin{bmatrix}}
\newcommand{\esqm}{\end{bmatrix}}
\newcommand{\bsm}{\left(\begin{smallmatrix}}
\newcommand{\esm}{\end{smallmatrix}\right)}
\newcommand{\pars}[1]{\left( #1 \right)}

\newcommand{\curlies}[1]{\left\{ #1 \right\}}

\newcommand{\bfone}{\mathbf{1}}

\renewcommand{\Re}{\operatorname{Re}}

\newtheorem{theorem}{Theorem}[section]
\newtheorem{lemma}[theorem]{Lemma}
\newtheorem{cor}[theorem]{Corollary}

\theoremstyle{definition}
\newtheorem{definition}[theorem]{Definition}
\newtheorem{example}[theorem]{Example}
\newtheorem{xca}[theorem]{Exercise}

\theoremstyle{remark}
\newtheorem{remark}[theorem]{Remark}

\numberwithin{equation}{section}

\newcommand{\abs}[1]{\lvert#1\rvert}

\newcommand{\blankbox}[2]{%
  \parbox{\columnwidth}{\centering
    \setlength{\fboxsep}{0pt}%
    \fbox{\raisebox{0pt}[#2]{\hspace{#1}}}%
  }%
}

\begin{document}

\title{Meromorphic continuation of the mean signature of fractional brownian motion}

\author{Andrew Ursitti}
\address{West Lafayette, Indiana}

\email{X@Y.org, X = andrew, Y = ursitti}


\date{\today}


\begin{abstract}
It is proved that the mean signature of multi-dimensional 
fractional brownian motion admits a meromorphic continuation in the 
hurst parameter to the entire complex plane. Each contstituent mean
iterated integral is a sum of 
hypergeometric integrals indexed by the pair partitions 
which refine the partition arising from the sequential list of 
integrands which defines it. Furthermore, 
each such hypergeometric integral is holomorphic 
in the complement of a finite union of rational progressions
determined by the combinatorial structure of the pair partition
which defines it.
It is not proved that these
singularities actually exist, it is only proved 
that the singularities are of finite order and they can only 
occur in the specified discrete set of rational numbers. 
\end{abstract}

\maketitle

\section*{Introduction}
Let $B=(B^1,\ldots, B^d)$ be a $d$-dimensional 
fractional brownian motion with hurst parameter $H\in(0,1)$ and 
let $\{V_1,\ldots, V_d\}$ be a list of $d$ smooth vector 
fields on $\R^n$, bounded and having bounded derivatives
of all orders. 
For fixed $H>1/2$ and $x\in\R^n$, let $X^{x}$ denote 
the $\R^n$-valued stochastic process defined by
\begin{equation}\label{asdfjkl;400}
X^{x}_t=x+\sum_{i=1}^d\int_0^tV_i(X_s^{x})dB^i_s.
\end{equation}
Here the vector field $V_i$ is viewed as a mapping $V_i:\R^n\to\R^n$.
Since $H>1/2$ is assumed, (\ref{asdfjkl;400}) can be 
defined using Young's theory of integration \cites{MR2604669,MR1555421}.

Associated with the fractional brownian motion $B$, and the vector 
fields $\{V_1,\ldots, V_d\}$,
one has the expectation operator
$\mathbf{P}_t^{x}:\mathscr{C}^\infty_c(\R^n;\C)\to\C$ given by 
$\mathbf{P}_t^{x}f=\mathbf{E}[f(X^{x}_t)]$. 
In \cite{MR2320949}, Baudoin and Coutin proved that there is a 
family $\{\Gamma_k^H\}_{k\geq 0}$ of differential operators on 
$\R^n$ such that for each $N$,
\[
\mathbf{P}^x_tf=\sum_{k=0}^Nt^{2kH}\Gamma_k^Hf(x)+o(t^{(2N+1)H})
\]
as $t\downarrow 0$. Moreover, they proved that 
\begin{equation}\label{asdfjkl;401}
\Gamma_k^H=\sum_{i_1,\ldots, i_{2k}\in\{1,\ldots, d\}}
\mathbf{E}\left[\int_{\Delta^{2k}[0,1]}dB^{i_1}_{t_1}\cdots dB^{i_{2k}}_{t_{2k}}\right]
V_{i_1}\cdots V_{i_{2k}}
\end{equation}
and 
\begin{align}
\mathbf{E}\left[\int_{\Delta^{2k}[0,1]} dB^{i_1}_{t_1}\cdots dB^{i_{2k}}_{t_{2k}}\right]& \notag\\
&\hspace{-3cm}=\frac{H^k(2H-1)^k}{k!2^k} \notag \\
&\hspace{-2.5cm}\times\sum_{\sigma\in \mathfrak{S}_{2k}}
\int_{\Delta^{2k}[0,1]}\prod_{l=1}^k\delta_{i_{\sigma(2l-1)},i_{\sigma(2l)}}
|s_{\sigma(2l)}-s_{\sigma(2l-1)}|^{2H-2}ds_1\wedge\ldots \wedge ds_{2k}.\label{asdfjkl;405}
\end{align}
where $\Delta^{2k}[0,1]=\{t\in\R^{2k}:0<t_1<\ldots <t_{2k}<1\}$ is the standard 
increasing simplex. To 
clarify this expression, observe that the 
integrals which define the summands are encoded by \emph{pair partitions}
of the set $\{1,\ldots, 2k\}$. In other words, 
if $\mathcal{P}=\{\{j_1^1,j_1^2\},\ldots,\{j_{k}^1,j_k^2\}\}$ is a pair 
partition of $\{1,\ldots, 2k\}$ then one can 
consider the hypergeometric integral
\begin{equation}\label{asdfjkl;402}
L(\mathcal{P};H)
=\int_{\Delta^{2k}[0,1]}\prod_{l=1}^k|s_{j_l^1}-s_{j_l^2}|^{2H-2}ds_1\wedge\ldots \wedge ds_{2k},
\end{equation}
and the summands in (\ref{asdfjkl;405}) are such integrals.
However, not all pair partitions
occur - the factors $\delta_{i_{\sigma(2l-1)},i_{\sigma(2l)}}$
force only those pair partitions which refine the partition implied by 
the defining word $(i_1,\ldots, i_{2k})$
to occur in the sum. By this we mean that the integrand
in a given summand of (\ref{asdfjkl;405}) is nonzero if and only if 
$i_{\sigma(2l-1)}=i_{\sigma(2l)}$ for every index $l\leq k$, but 
the word $(i_1,\ldots, i_{2k})\in \{1,\ldots, d\}^{2k}$ already 
partitions the set $\{1,\ldots,2k\}$ according to the level sets 
of the map $p\mapsto i_p$, so the integrand of a given summand 
in (\ref{asdfjkl;405}) is nonzero if and only 
if the permutation $\sigma$ maps adjacent indices of the 
form $\{2l-1,2l\}$ into the 
blocks of the partition implied by the word 
$(i_1,\ldots, i_{2k})$. 

For example, suppose $d=6$, $k=5$, $(i_1,\ldots, i_{10})=(6,3,1,3,6,6,1,5,6,5)$ and
consider the four diagrams in Figure \ref{asdfjkl;figure1}.
The coloring scheme indicates the partition defined by the 
word $(6,3,1,3,6,6,1,5,6,5)$, i.e. distinct numbers have a 
common color if and only if they belong to the same level set of 
the map $(1,2,3,4,5,6,7,8,9,10)\mapsto (6,3,1,3,6,6,1,5,6,5)$, 
or equivalently if they belong to the same block of the implied 
partition. The bracketings indicate pair partitions, or 
equivalently elements of the quotient $\mathfrak{S}_{2k}/M_{2k}$, 
where $M_{2k}\subset \mathfrak{S}_{2k}$ is the abelian subgroup of 
order $2^k$ generated by the adjacency transpositions
$(1,2),(3,4),\ldots,(2k-1,2k)$.

\begin{figure}[h]
\begin{center}
\begin{pspicture}(0,1cm)
\rnode{N1}{\psframebox*[fillstyle=solid,fillcolor=red]{1}}
\rnode{N2}{\psframebox*[fillstyle=solid,fillcolor=green]{2}}
\rnode{N3}{\psframebox*[fillstyle=solid,fillcolor=cyan]{3}}
\rnode{N4}{\psframebox*[fillstyle=solid,fillcolor=green]{4}}
\rnode{N5}{\psframebox*[fillstyle=solid,fillcolor=red]{5}}
\rnode{N5.5}{}
\rnode{N6}{\psframebox*[fillstyle=solid,fillcolor=red]{6}}
\rnode{N7}{\psframebox*[fillstyle=solid,fillcolor=cyan]{7}}
\rnode{N8}{\psframebox*[fillstyle=solid,fillcolor=yellow]{8}}
\rnode{N9}{\psframebox*[fillstyle=solid,fillcolor=red]{9}}
\rnode{N10}{\psframebox*[fillstyle=solid,fillcolor=yellow]{10}}
\nput*[labelsep=4mm]{-90}{N5.5}{\fontsize{8pt}{8pt}$\mathcal{P}_1=\{\{1,2\},\{3,4\},\{5,6\},\{7,8\},\{9,10\}\}$}

\ncbar[angleA=90,arm=1mm,nodesep=2pt]{-}{N1}{N2}
\ncbar[angleA=90,arm=1mm,nodesep=2pt]{-}{N3}{N4}
\ncbar[angleA=90,arm=1mm,nodesep=2pt]{-}{N5}{N6}
\ncbar[angleA=90,arm=1mm,nodesep=2pt]{-}{N7}{N8}
\ncbar[angleA=90,arm=1mm,nodesep=2pt]{-}{N9}{N10}
\hspace{1cm}

\rnode{N01}{\psframebox*[fillstyle=solid,fillcolor=red]{1}}
\rnode{N02}{\psframebox*[fillstyle=solid,fillcolor=green]{2}}
\rnode{N03}{\psframebox*[fillstyle=solid,fillcolor=cyan]{3}}
\rnode{N04}{\psframebox*[fillstyle=solid,fillcolor=green]{4}}
\rnode{N05}{\psframebox*[fillstyle=solid,fillcolor=red]{5}}
\rnode{N05.5}{}
\rnode{N06}{\psframebox*[fillstyle=solid,fillcolor=red]{6}}
\rnode{N07}{\psframebox*[fillstyle=solid,fillcolor=cyan]{7}}
\rnode{N08}{\psframebox*[fillstyle=solid,fillcolor=yellow]{8}}
\rnode{N09}{\psframebox*[fillstyle=solid,fillcolor=red]{9}}
\rnode{N010}{\psframebox*[fillstyle=solid,fillcolor=yellow]{10}}
\nput*[labelsep=4mm]{-90}{N05.5}{\fontsize{8pt}{8pt}$\mathcal{P}_2=\{\{1,6\},\{2,4\},\{3,7\},\{5,9\},\{8,10\}\}$}

\ncbar[angleA=90,arm=1mm,nodesep=2pt]{-}{N01}{N06}
\ncbar[angleA=90,arm=2mm,nodesep=2pt]{-}{N02}{N04}
\ncbar[angleA=90,arm=3mm,nodesep=2pt]{-}{N03}{N07}
\ncbar[angleA=90,arm=2mm,nodesep=2pt]{-}{N05}{N09}
\ncbar[angleA=90,arm=1mm,nodesep=2pt]{-}{N08}{N010}
\end{pspicture}

\begin{pspicture}(0,1.65cm)
\rnode{N1}{\psframebox*[fillstyle=solid,fillcolor=red]{1}}
\rnode{N2}{\psframebox*[fillstyle=solid,fillcolor=green]{2}}
\rnode{N3}{\psframebox*[fillstyle=solid,fillcolor=cyan]{3}}
\rnode{N4}{\psframebox*[fillstyle=solid,fillcolor=green]{4}}
\rnode{N5}{\psframebox*[fillstyle=solid,fillcolor=red]{5}}
\rnode{N5.5}{}
\rnode{N6}{\psframebox*[fillstyle=solid,fillcolor=red]{6}}
\rnode{N7}{\psframebox*[fillstyle=solid,fillcolor=cyan]{7}}
\rnode{N8}{\psframebox*[fillstyle=solid,fillcolor=yellow]{8}}
\rnode{N9}{\psframebox*[fillstyle=solid,fillcolor=red]{9}}
\rnode{N10}{\psframebox*[fillstyle=solid,fillcolor=yellow]{10}}
\nput*[labelsep=4mm]{-90}{N5.5}{\fontsize{8pt}{8pt}$\mathcal{P}_3=\{\{1,9\},\{2,4\},\{3,7\},\{5,6\},\{8,10\}\}$}

\ncbar[angleA=90,arm=3mm,nodesep=2pt]{-}{N1}{N9}
\ncbar[angleA=90,arm=1mm,nodesep=2pt]{-}{N2}{N4}
\ncbar[angleA=90,arm=2mm,nodesep=2pt]{-}{N3}{N7}
\ncbar[angleA=90,arm=1mm,nodesep=2pt]{-}{N5}{N6}
\ncbar[angleA=90,arm=2mm,nodesep=2pt]{-}{N8}{N10}

\hspace{1cm}

\rnode{N01}{\psframebox*[fillstyle=solid,fillcolor=red]{1}}
\rnode{N02}{\psframebox*[fillstyle=solid,fillcolor=green]{2}}
\rnode{N03}{\psframebox*[fillstyle=solid,fillcolor=cyan]{3}}
\rnode{N04}{\psframebox*[fillstyle=solid,fillcolor=green]{4}}
\rnode{N05}{\psframebox*[fillstyle=solid,fillcolor=red]{5}}
\rnode{N05.5}{}
\rnode{N06}{\psframebox*[fillstyle=solid,fillcolor=red]{6}}
\rnode{N07}{\psframebox*[fillstyle=solid,fillcolor=cyan]{7}}
\rnode{N08}{\psframebox*[fillstyle=solid,fillcolor=yellow]{8}}
\rnode{N09}{\psframebox*[fillstyle=solid,fillcolor=red]{9}}
\rnode{N010}{\psframebox*[fillstyle=solid,fillcolor=yellow]{10}}
\nput*[labelsep=4mm]{-90}{N05.5}{\fontsize{8pt}{8pt}$\mathcal{P}_4=\{\{1,9\},\{2,7\},\{3,4\},\{5,6\},\{8,10\}\}$}

\ncbar[angleA=90,arm=3mm,nodesep=2pt]{-}{N01}{N09}
\ncbar[angleA=90,arm=2mm,nodesep=2pt]{-}{N02}{N07}
\ncbar[angleA=90,arm=1mm,nodesep=2pt]{-}{N03}{N04}
\ncbar[angleA=90,arm=1mm,nodesep=2pt]{-}{N05}{N06}
\ncbar[angleA=90,arm=1mm,nodesep=2pt]{-}{N08}{N010}
\end{pspicture}
\vspace{.5cm}
\caption{}\label{asdfjkl;figure1}
\vspace{1mm}
\end{center}
\end{figure}

Evidently $\mathcal{P}_2$ and $\mathcal{P}_3$ 
both
refine the level set partition 
of the map 
\[
(1,2,3,4,5,6,7,8,9,10)\mapsto (6,3,1,3,6,6,1,5,6,5),
\]
so $L(\mathcal{P}_2,H)$ and $L(\mathcal{P}_3,H)$ as defined in 
(\ref{asdfjkl;402}) each occur with multiplicity $2^k=|M_{2k}|$ 
in the sum (\ref{asdfjkl;405}). However, 
$\mathcal{P}_1=1M_{2k}\in \mathfrak{S}_{2k}/M_{2k}$ and 
$\mathcal{P}_4$ \emph{do not} refine the aforementioned 
level set partition, as 
can be easily seen from the coloring scheme since the constituent 
pairs $\{1,2\},\{3,4\},\{7,8\},\{9,10\}$ in $\mathcal{P}_1$ and
$\{3,4\},\{2,7\}$ in $\mathcal{P}_2$ connect distinct 
colors. Therefore, $L(\mathcal{P}_2,H)$ and $L(\mathcal{P}_3,H)$
do not occur in the sum (\ref{asdfjkl;405}). 

Refinement of partitions 
will be abbreviated with the ``$\leq$" symbol, so the statement 
$\{j_l^1,j_l^2\}\leq (i_1,\ldots, i_{2k})$ is true 
if and only if $i_{j^1_l}=i_{j_l^2}$ and $\mathcal{P}\leq (i_1,\ldots, i_{2k})$
if each constituent pair of $\mathcal{P}$ is $\leq(i_1,\ldots, i_{2k})$. 
With this notation
we can rewrite (\ref{asdfjkl;405}) as 
\begin{equation}
\mathbf{E}\left[\int_{\Delta^{2k}[0,1]} dB^{i_1}_{t_1}\cdots dB^{i_{2k}}_{t_{2k}}\right]=\frac{H^k(2H-1)^k}{k!} 
\sum_{\mathcal{P}\leq (i_1,\ldots, i_{2k})}
L(\mathcal{P};H).\label{asdfjkl;406}
\end{equation}

A natural question to ask is that of the 
possibility of the meromorphic continuation 
of the operators $\mathbf{P}_t^x$ to a complex 
neighborhood of the
interval $H\in(1/2,1)$, even for those $H$ at which
the stochastic 
integral (\ref{asdfjkl;400}) is ill-posed. With this
in mind, the authors
of \cite{MR2320949} conjectured that the coefficients 
of the constituent operators 
$\Gamma_k^H$ for $k\geq 2$ have meromorphic continuations
with poles in the set $\{1/2j:2\leq j\leq k\}$. In this 
note we prove that the $\Gamma_k^H$ 
have meromorphic continuations to the entire complex 
plane in the variable $H$, and we show that poles can 
only occur in certain specific sets of rational numbers 
which depend on the combinatorial structure of the 
words which define the constituent iterated integrals.
For $\Re H>0$, our results indicate that there may be
more poles than those which are conjectured to 
exist in \cite{MR2320949}. However, none of these 
poles is actually shown to exist - the laurent series 
coefficients are still prohibitively 
complicated so as of yet we cannot prove that they're nonzero.

For any  
pair $\{j^1,j^2\}\subset \{1,\ldots 2k\}$, define
\[
I(\{j^1,j^2\})=[j^1\wedge j^2+1,j^1\vee j^2]
=\{j^1\wedge j^2+1,\ldots, j^1\vee j^2\}\subset\{2,\ldots, 2k\}.
\]
For example, if 
$k=5$ then $I(\{10,6\})=[7,10]=\{7,8,9,10\}$,
$I(\{7,3\})=[4,7]=\{4,5,6,7\}$ and $I(\{1,2\})=[2]=\{2\}$\footnote{Here
and below we 
use the notation $[n]=[n,n]=\{n\}$ for an interval of natural numbers
with only one element}. The 
significance of $I(\{j^1,j^2\})$ is that if without loss of 
generality $j^1<j^2$ then
\[
s_{j^2}-s_{j^1}=(s_{j^2}-s_{j^2-1})+(s_{j^2-1}-s_{j^2-2})+\ldots
+(s_{j^1+1}-s_{j^1})
\]
and $I(\{j^1,j^2\})$ represents the subset of variables $\{x_1,\ldots, x_{2k}\}$
in the image of this sum under the change of variable $x_i=s_i-s_{i-1}$ $(s_0=0)$.

For any pair partition $\mathcal{P}=\{\{j_1^1,j_1^2\},\ldots,\{j_{k}^1,j_k^2\}\}$, 
the map $I$ can be applied
to each of the pairs in $\mathcal{P}$, thus obtaining a map
\[
\mathcal{I}:\{\text{pair partitions of }\{1,\ldots, 2k\}\}\longrightarrow
\{\text{sets of }k\text{ subintervals of }\{2,\ldots, 2k\}\}.
\]
For example, if $k=3$ and $\mathcal{P}_1=\{\{4,6\},\{5,2\},\{1,3\}\}$ then 
$\mathcal{I}(\mathcal{P}_1)=\{[5,6],[3,5],[2,3]\}$.
On the other hand if $\mathcal{P}_2=\{\{1,6\},\{2,5\},\{3,4\}\}$ then 
$\mathcal{I}(\mathcal{P}_2)=\{[2,6],[3,5],[4]\}$,
and if $\mathcal{P}_3=\{\{1,4\},\{2,5\},\{3,6\}\}$ then 
$\mathcal{I}(\mathcal{P}_3)=\{[2,4],[3,5],[4,6]\}$.
Now for any pair partition $\mathcal{P}$ of $\{1,\ldots, 2k\}$ and 
any subset $S\subset \{1,\ldots,2k\}$ define 
\[
[S|\mathcal{P}]=|\mathfrak{P}(S)\cap\mathcal{I}(\mathcal{P})|
=\text{number of subsets of }S\text{ which are elements of }\mathcal{I}(\mathcal{P}).
\]

\begin{theorem}\label{asdfjkl;404}
For any pair partition $\mathcal{P}=\{\{j_1^1,j_1^2\},\ldots,\{j_k^1,j_k^2\}\}$
of the set $[1,2k]$, the 
hypergeometric integral $L(\mathcal{P};H)$
has a meromorphic continuation in the variable $H$ to the 
entire plane $\C$ which is holomorphic 
in the complement of the following union of
rational progressions:
\[
\curlies{1-\frac{|S|+l}{2[S|\mathcal{P}]}:S\subset [1, 2k] ,[S|\mathcal{P}]>0,l=0,1,2,\ldots}\subset\mathbf{Q}.
\]
\end{theorem}

\begin{cor}\label{asdfjkl;414}
For any word $(i_1,\ldots, i_{2k})\in \{1,\ldots,d\}^{2k}$, the 
sum 
\[
\frac{k!}{H^k(2H-1)^k}\mathbf{E}\left[\int_{\Delta^{2k}[0,1]} dB^{i_1}_{t_1}\cdots dB^{i_{2k}}_{t_{2k}}\right]=
\sum_{\mathcal{P}\leq (i_1,\ldots, i_{2k})}
L(\mathcal{P};H)
\]
has a meromorphic continuation in the variable $H$ to the 
entire plane $\C$ which is holomorphic 
in the complement of the following union of
rational progressions:
\[
\curlies{1-\frac{|S|+l}{2[S|\mathcal{P}]}:S\subset [1, 2k], \mathcal{P}\leq (i_1,\ldots, i_{2k}) ,[S|\mathcal{P}]>0,l=0,1,2,\ldots}\subset\mathbf{Q}.
\]
\end{cor}

The corollary follows immediately from the theorem 
by way of the equality (\ref{asdfjkl;406}).
We stress yet again that since we've not proved 
that the meromorphic continuations are unbounded 
in every neighborhood of a point in the progression 
$1-(|S|+l)/2[S|\mathcal{P}]$, these numbers 
are merely the \emph{only candidates}
for the poles.

In order to understand the structure of
these rational progressions we
observe that 
$[S|\mathcal{P}]$ is additive 
with respect to any pairwise monotone and 
nonadjacent disjoint decomposition. 
By this we mean that if $S=\cup_{r=1}^s S_r$ and 
for $1\leq r<\rho\leq s$ the greatest 
element of $S_r$ is at least two less than 
the least element of $S_\rho$ then 
$[S|\mathcal{P}]=\sum_{r=1}^s[S_r|\mathcal{P}]$. 
This is easy to see, because the existence 
of a nontrivial gap between $S_r$ and $S_\rho$ 
means that any subinterval of $[1,2k]$ contained 
in $S_r\cup S_\rho$ must be contained in 
either $S_r$ or in $S_\rho$. 
In particular this is true of the maximal 
interval decomposition of $S$, i.e. 
$S=\cup_{r=1}^s I_r$ where each $I_r$ is 
an interval and if $1\leq r<\rho\leq s$ then 
the right endpoint of $I_r$ is at least two 
less than the left endpoint of $I_\rho$. 

Thus, apparently $[S|\mathcal{P}]=\sum_{r=1}^s [I_r|\mathcal{P}]$
so it is enough to understand the expression
$[I|\mathcal{P}]$ where $I$ is an interval. For this, 
we define two functions on the subintervals of $[1,2k]$
taking values in the subsets of $[1,2k]$,
the $\mathcal{P}$ \emph{augmentation} of $I$:
\[
\Aug_\mathcal{P}(I)=
\begin{cases}
I\cup\{\text{the element immediately left adjacent to } I\} \\
\quad \text{if the element immediately left adjacent to } I \\
\quad \text{is paired by }\mathcal{P} \text{ with an element of }I \\
I \text{ in all other cases},
\end{cases}
\]
and the $\mathcal{P}$ \emph{deficiency} of $I$:
\[
\Def_\mathcal{P}(I)=
\curlies{\text{elements of }I\text{ which are \emph{not} paired by }
\mathcal{P}\text{ with elements of } \Aug_\mathcal{P}(I)}.
\]

With these definitions in place, it is straightforward 
to see that $\Aug_\mathcal{P}(I)\setminus \Def_\mathcal{P}(I)$
is precisely the union of $\mathcal{P}$-pairs $\{j^1,j^2\}$ such 
that $I(\{j^1,j^2\})\subset I$, and therefore
\[
2[I|\mathcal{P}]=|\Aug_\mathcal{P}(I)|-|\Def_\mathcal{P}(I)|.
\]
Consequently, 
\[
2[S|\mathcal{P}]=\sum_{r=1}^s|\Aug_\mathcal{P}(I_r)|-|\Def_\mathcal{P}(I_r)|
\]
where $S=\cup_{r=1}^s I_r$ still denotes the maximal interval decomposition 
of $S$. The following corollary follows immediately 
from Theorem \ref{asdfjkl;404} by decomposing all subsets of $[1,2k]$ 
into maximal connected intervals.

\begin{cor}\label{asdfjkl;1404}
For any pair partition $\mathcal{P}=\{\{j_1^1,j_1^2\},\ldots,\{j_k^1,j_k^2\}\}$
of the set $[1, 2k]$, the 
hypergeometric integral $L(\mathcal{P};H)$
has a meromorphic continuation in the variable $H$ to the 
entire plane $\C$ which is holomorphic 
in the complement of the union of 
rational progressions of the form
\begin{equation}\label{asdfjkl;1000}
\curlies{1-\frac{\sum_{r=1}^s|I_r|+l}{\sum_{r=1}^s|\Aug_\mathcal{P}(I_r)|-|\Def_\mathcal{P}(I_r)|}:l=0,1,2,\ldots}\subset\mathbf{Q}
\end{equation}
with $\{I_1,\ldots, I_s\}$ equal to any 
collection of pairwise nonadjacent subintervals of $[1,2k]$ 
such that the denominator is nonzero.
\end{cor}

The size $|I_r|$ of any of the intervals in Corollary \ref{asdfjkl;1404}
can be extracted from the expression in the denominator
$|\Aug_\mathcal{P}(I_r)|-|\Def_\mathcal{P}(I_r)|$ by observing 
that $|\Aug_\mathcal{P}(I_r)|$ is either equal to $|I_r|$ or 
$|I_r|+1$ depending on whether or not $\Aug_\mathcal{P}(I_r)=I_r$ 
or not. In the latter case we say that $I_r$ is $\mathcal{P}$-\emph{augmented}, 
and 
\[
\sum_{r=1}^s|\Aug_\mathcal{P}(I_r)|=\sum_{r=1}^s|I_r|+
(\text{number of }\mathcal{P} -\text{augmented intervals in } \{I_1,\ldots ,I_s\}).
\]
Likewise, an element of $\Def_\mathcal{P}(I_r)$ is said to be $\mathcal{P}$-deficient
and 
\[
\sum_{r=1}^s|\Def_\mathcal{P}(I_r)|=
(\text{number of }\mathcal{P} -\text{deficient points in } \cup_{r=1}^s I_r\}).
\]
Therefore, by decomposing an arbitrary subset $S\subset [1,2k]$ uniquely 
as a union $S=\cup_{r=1}^s I_r$ of pairwise nonadjacent intervals and defining the 
$\mathcal{P}$-\emph{deficient} points of $S$ as the union of those in the $I_r$
individually, the expression in (\ref{asdfjkl;1000})
can be rewritten as
\begin{equation}\label{asdfjkl;1001}
1-\frac{|S|+l}{\pars{\begin{array}{c}|S|+(\text{number of }\mathcal{P} -\text{augmented components in } S)\\
-(\text{number of }\mathcal{P}-\text{deficient points in }S)\end{array}}}.
\end{equation}

Note that for any $S$ and any $\mathcal{P}$, 
the number of $\mathcal{P}$-augmented components 
in $S$ is at most $|S|$ and for this reason the fraction 
written above is not less than $1/2$.
With this in mind, the following corollary follows immediately 
from Theorem \ref{asdfjkl;404} and 
Corollary \ref{asdfjkl;1404}.

\begin{cor}\label{asdfjkl;407}
For any pair partition $\mathcal{P}=\{\{j_1^1,j_1^2\},\ldots,\{j_k^1,j_k^2\}\}$
of the set $[1,2k]$, the meromorphic continuation of the
hypergeometric integral $L(\mathcal{P};H)$ is holomorphic at all nonreal points and all
real points greater than the maximal value of (\ref{asdfjkl;1001}) over all 
subsets $S\subset [1,2k]$ such that the denominator is nonzero. In 
particular it is holomorphic at all points greater than $1/2$ and 
the only possible poles in the interval $(0,1/2]$ are at those values 
of (\ref{asdfjkl;1404}) arising from the subsets $S\subset [1,2k]$ such 
that the number of $\mathcal{P}$-augmented components in $S$ 
exceeds the number of $\mathcal{P}$-deficiencies in $S$.
\end{cor}

\begin{figure}
\begin{center}
\begin{pspicture}(0,1cm)

\rnode{NA1}{\psframebox*[linewidth=.5mm,fillstyle=solid,fillcolor=red]  {1}} 
\rnode{NA2}{\psframebox[linewidth=.5mm,fillstyle=solid,fillcolor=green]  {2}}
\rnode{NA3}{\psframebox[linewidth=.5mm,fillstyle=solid,fillcolor=cyan]  {3}}
\rnode{NA4}{\psframebox[linewidth=.5mm,fillstyle=solid,fillcolor=yellow]  {4}}
\rnode{NA5}{\psframebox[linewidth=.5mm,fillstyle=solid,fillcolor=cyan]  {5}}
\rnode{NA6}{\psframebox[linewidth=.5mm,fillstyle=solid,fillcolor=yellow]  {6}} 
\rnode{NA7}{\psframebox[linewidth=.5mm,fillstyle=solid,fillcolor=red]  {7}}
\rnode{NA8}{\psframebox[linewidth=.5mm,fillstyle=solid,fillcolor=green]  {8}}
\rnode{NA9}{\psframebox*[linewidth=.5mm,fillstyle=solid,fillcolor=cyan]  {9}}
\rnode{NA9.5}{}
\rnode{NA10}{\psframebox[linewidth=.5mm,fillstyle=solid,fillcolor=violet] {10}}
\rnode{NA11}{\psframebox[linewidth=.5mm,fillstyle=solid,fillcolor=cyan] {11}}
\rnode{NA12}{\psframebox*[linewidth=.5mm,fillstyle=solid,fillcolor=olive] {12}}
\rnode{NA13}{\psframebox[linewidth=.5mm,fillstyle=solid,fillcolor=red] {13}}
\rnode{NA14}{\psframebox[linewidth=.5mm,fillstyle=solid,fillcolor=red] {14}}
\rnode{NA15}{\psframebox[linewidth=.5mm,fillstyle=solid,fillcolor=violet] {15}}
\rnode{NA16}{\psframebox[linewidth=.5mm,fillstyle=solid,fillcolor=violet] {16}}
\rnode{NA17}{\psframebox[linewidth=.5mm,fillstyle=solid,fillcolor=olive] {17}}
\rnode{NA18}{\psframebox*[linewidth=.5mm,fillstyle=solid,fillcolor=violet] {18}}

\nput*[labelsep=4mm]{-90}{NA9.5}{\fontsize{8pt}{8pt}$S=[2,8]\cup[10,11]\cup[13,17]\quad|S|=14\quad2[S|\mathcal{P}]=|S|+3-1=16$}

\ncbar[angleA=90,arm=1mm,nodesep=2pt]{-}{NA4}{NA6}
\ncbar[angleA=90,arm=1mm,nodesep=2pt]{-}{NA13}{NA14}
\ncbar[angleA=90,arm=1mm,nodesep=2pt]{-}{NA9}{NA11}
\ncbar[angleA=90,arm=1mm,nodesep=2pt]{-}{NA15}{NA16}

\ncbar[angleA=90,arm=2mm,nodesep=2pt]{-}{NA3}{NA5}
\ncbar[angleA=90,arm=2mm,nodesep=2pt]{-}{NA12}{NA17}

\ncbar[angleA=90,arm=3mm,nodesep=2pt]{-}{NA2}{NA8}
\ncbar[angleA=90,arm=3mm,nodesep=2pt]{-}{NA10}{NA18}

\ncbar[angleA=90,arm=4mm,nodesep=2pt]{-}{NA1}{NA7}
\end{pspicture}

\begin{pspicture}(0,2cm)

\rnode{NA1}{\psframebox*[linewidth=.5mm,fillstyle=solid,fillcolor=red]  {1}} 
\rnode{NA2}{\psframebox*[linewidth=.5mm,fillstyle=solid,fillcolor=green]  {2}}
\rnode{NA3}{\psframebox[linewidth=.5mm,fillstyle=solid,fillcolor=cyan]  {3}}
\rnode{NA4}{\psframebox[linewidth=.5mm,fillstyle=solid,fillcolor=yellow]  {4}}
\rnode{NA5}{\psframebox*[linewidth=.5mm,fillstyle=solid,fillcolor=cyan]  {5}}
\rnode{NA6}{\psframebox[linewidth=.5mm,fillstyle=solid,fillcolor=yellow]  {6}} 
\rnode{NA7}{\psframebox[linewidth=.5mm,fillstyle=solid,fillcolor=red]  {7}}
\rnode{NA8}{\psframebox[linewidth=.5mm,fillstyle=solid,fillcolor=green]  {8}}
\rnode{NA9}{\psframebox[linewidth=.5mm,fillstyle=solid,fillcolor=cyan]  {9}}
\rnode{NA9.5}{}
\rnode{NA10}{\psframebox[linewidth=.5mm,fillstyle=solid,fillcolor=violet] {10}}
\rnode{NA11}{\psframebox[linewidth=.5mm,fillstyle=solid,fillcolor=cyan] {11}}
\rnode{NA12}{\psframebox*[linewidth=.5mm,fillstyle=solid,fillcolor=olive] {12}}
\rnode{NA13}{\psframebox[linewidth=.5mm,fillstyle=solid,fillcolor=red] {13}}
\rnode{NA14}{\psframebox[linewidth=.5mm,fillstyle=solid,fillcolor=red] {14}}
\rnode{NA15}{\psframebox*[linewidth=.5mm,fillstyle=solid,fillcolor=violet] {15}}
\rnode{NA16}{\psframebox*[linewidth=.5mm,fillstyle=solid,fillcolor=violet] {16}}
\rnode{NA17}{\psframebox[linewidth=.5mm,fillstyle=solid,fillcolor=olive] {17}}
\rnode{NA18}{\psframebox[linewidth=.5mm,fillstyle=solid,fillcolor=violet] {18}}

\nput*[labelsep=4mm]{-90}{NA9.5}{\fontsize{8pt}{8pt}$S=[3,4]\cup[6,11]\cup[13,14]\cup[17,18]\quad|S|=12\quad2[S|\mathcal{P}]=|S|+0-8=4$}

\ncbar[angleA=90,arm=1mm,nodesep=2pt]{-}{NA4}{NA6}
\ncbar[angleA=90,arm=1mm,nodesep=2pt]{-}{NA13}{NA14}
\ncbar[angleA=90,arm=1mm,nodesep=2pt]{-}{NA9}{NA11}
\ncbar[angleA=90,arm=1mm,nodesep=2pt]{-}{NA15}{NA16}

\ncbar[angleA=90,arm=2mm,nodesep=2pt]{-}{NA3}{NA5}
\ncbar[angleA=90,arm=2mm,nodesep=2pt]{-}{NA12}{NA17}

\ncbar[angleA=90,arm=3mm,nodesep=2pt]{-}{NA2}{NA8}
\ncbar[angleA=90,arm=3mm,nodesep=2pt]{-}{NA10}{NA18}

\ncbar[angleA=90,arm=4mm,nodesep=2pt]{-}{NA1}{NA7}
\end{pspicture}

\begin{pspicture}(0,2cm)

\rnode{NA1}{\psframebox[linewidth=.5mm,fillstyle=solid,fillcolor=red]  {1}} 
\rnode{NA2}{\psframebox[linewidth=.5mm,fillstyle=solid,fillcolor=green]  {2}}
\rnode{NA3}{\psframebox[linewidth=.5mm,fillstyle=solid,fillcolor=cyan]  {3}}
\rnode{NA4}{\psframebox*[linewidth=.5mm,fillstyle=solid,fillcolor=yellow]  {4}}
\rnode{NA5}{\psframebox[linewidth=.5mm,fillstyle=solid,fillcolor=cyan]  {5}}
\rnode{NA6}{\psframebox[linewidth=.5mm,fillstyle=solid,fillcolor=yellow]  {6}} 
\rnode{NA7}{\psframebox*[linewidth=.5mm,fillstyle=solid,fillcolor=red]  {7}}
\rnode{NA8}{\psframebox[linewidth=.5mm,fillstyle=solid,fillcolor=green]  {8}}
\rnode{NA9}{\psframebox[linewidth=.5mm,fillstyle=solid,fillcolor=cyan]  {9}}
\rnode{NA9.5}{}
\rnode{NA10}{\psframebox*[linewidth=.5mm,fillstyle=solid,fillcolor=violet] {10}}
\rnode{NA11}{\psframebox*[linewidth=.5mm,fillstyle=solid,fillcolor=cyan] {11}}
\rnode{NA12}{\psframebox[linewidth=.5mm,fillstyle=solid,fillcolor=olive] {12}}
\rnode{NA13}{\psframebox*[linewidth=.5mm,fillstyle=solid,fillcolor=red] {13}}
\rnode{NA14}{\psframebox[linewidth=.5mm,fillstyle=solid,fillcolor=red] {14}}
\rnode{NA15}{\psframebox*[linewidth=.5mm,fillstyle=solid,fillcolor=violet] {15}}
\rnode{NA16}{\psframebox[linewidth=.5mm,fillstyle=solid,fillcolor=violet] {16}}
\rnode{NA17}{\psframebox*[linewidth=.5mm,fillstyle=solid,fillcolor=olive] {17}}
\rnode{NA18}{\psframebox[linewidth=.5mm,fillstyle=solid,fillcolor=violet] {18}}

\nput*[labelsep=4mm]{-90}{NA9.5}{\fontsize{8pt}{8pt}$S=[1,3]\cup[5,9]\cup[12]\cup[14]\cup[16]\cup[18] \quad|S|=11\quad2[S|\mathcal{P}]=|S|+3-8=6$}

\ncbar[angleA=90,arm=1mm,nodesep=2pt]{-}{NA4}{NA6}
\ncbar[angleA=90,arm=1mm,nodesep=2pt]{-}{NA13}{NA14}
\ncbar[angleA=90,arm=1mm,nodesep=2pt]{-}{NA9}{NA11}
\ncbar[angleA=90,arm=1mm,nodesep=2pt]{-}{NA15}{NA16}

\ncbar[angleA=90,arm=2mm,nodesep=2pt]{-}{NA3}{NA5}
\ncbar[angleA=90,arm=2mm,nodesep=2pt]{-}{NA12}{NA17}

\ncbar[angleA=90,arm=3mm,nodesep=2pt]{-}{NA2}{NA8}
\ncbar[angleA=90,arm=3mm,nodesep=2pt]{-}{NA10}{NA18}

\ncbar[angleA=90,arm=4mm,nodesep=2pt]{-}{NA1}{NA7}
\end{pspicture}

\begin{pspicture}(0,2cm)

\rnode{NA1}{\psframebox*[linewidth=.5mm,fillstyle=solid,fillcolor=red]  {1}} 
\rnode{NA2}{\psframebox*[linewidth=.5mm,fillstyle=solid,fillcolor=green]  {2}}
\rnode{NA3}{\psframebox*[linewidth=.5mm,fillstyle=solid,fillcolor=cyan]  {3}}
\rnode{NA4}{\psframebox[linewidth=.5mm,fillstyle=solid,fillcolor=yellow]  {4}}
\rnode{NA5}{\psframebox[linewidth=.5mm,fillstyle=solid,fillcolor=cyan]  {5}}
\rnode{NA6}{\psframebox[linewidth=.5mm,fillstyle=solid,fillcolor=yellow]  {6}} 
\rnode{NA7}{\psframebox*[linewidth=.5mm,fillstyle=solid,fillcolor=red]  {7}}
\rnode{NA8}{\psframebox*[linewidth=.5mm,fillstyle=solid,fillcolor=green]  {8}}
\rnode{NA9}{\psframebox*[linewidth=.5mm,fillstyle=solid,fillcolor=cyan]  {9}}
\rnode{NA9.5}{}
\rnode{NA10}{\psframebox*[linewidth=.5mm,fillstyle=solid,fillcolor=violet] {10}}
\rnode{NA11}{\psframebox*[linewidth=.5mm,fillstyle=solid,fillcolor=cyan] {11}}
\rnode{NA12}{\psframebox*[linewidth=.5mm,fillstyle=solid,fillcolor=olive] {12}}
\rnode{NA13}{\psframebox*[linewidth=.5mm,fillstyle=solid,fillcolor=red] {13}}
\rnode{NA14}{\psframebox[linewidth=.5mm,fillstyle=solid,fillcolor=red] {14}}
\rnode{NA15}{\psframebox*[linewidth=.5mm,fillstyle=solid,fillcolor=violet] {15}}
\rnode{NA16}{\psframebox[linewidth=.5mm,fillstyle=solid,fillcolor=violet] {16}}
\rnode{NA17}{\psframebox*[linewidth=.5mm,fillstyle=solid,fillcolor=olive] {17}}
\rnode{NA18}{\psframebox*[linewidth=.5mm,fillstyle=solid,fillcolor=violet] {18}}

\nput*[labelsep=4mm]{-90}{NA9.5}{\fontsize{8pt}{8pt}$S=[4,6]\cup[14]\cup[16]\quad|S|=5\quad2[S|\mathcal{P}]=|S|+3-0=8$}

\ncbar[angleA=90,arm=1mm,nodesep=2pt]{-}{NA4}{NA6}
\ncbar[angleA=90,arm=1mm,nodesep=2pt]{-}{NA13}{NA14}
\ncbar[angleA=90,arm=1mm,nodesep=2pt]{-}{NA9}{NA11}
\ncbar[angleA=90,arm=1mm,nodesep=2pt]{-}{NA15}{NA16}

\ncbar[angleA=90,arm=2mm,nodesep=2pt]{-}{NA3}{NA5}
\ncbar[angleA=90,arm=2mm,nodesep=2pt]{-}{NA12}{NA17}

\ncbar[angleA=90,arm=3mm,nodesep=2pt]{-}{NA2}{NA8}
\ncbar[angleA=90,arm=3mm,nodesep=2pt]{-}{NA10}{NA18}

\ncbar[angleA=90,arm=4mm,nodesep=2pt]{-}{NA1}{NA7}
\end{pspicture}

\begin{pspicture}(0,2cm)

\rnode{NA1}{\psframebox*[linewidth=.5mm,fillstyle=solid,fillcolor=red]  {1}} 
\rnode{NA2}{\psframebox[linewidth=.5mm,fillstyle=solid,fillcolor=green]  {2}}
\rnode{NA3}{\psframebox[linewidth=.5mm,fillstyle=solid,fillcolor=cyan]  {3}}
\rnode{NA4}{\psframebox[linewidth=.5mm,fillstyle=solid,fillcolor=yellow]  {4}}
\rnode{NA5}{\psframebox[linewidth=.5mm,fillstyle=solid,fillcolor=cyan]  {5}}
\rnode{NA6}{\psframebox[linewidth=.5mm,fillstyle=solid,fillcolor=yellow]  {6}} 
\rnode{NA7}{\psframebox[linewidth=.5mm,fillstyle=solid,fillcolor=red]  {7}}
\rnode{NA8}{\psframebox*[linewidth=.5mm,fillstyle=solid,fillcolor=green]  {8}}
\rnode{NA9}{\psframebox*[linewidth=.5mm,fillstyle=solid,fillcolor=cyan]  {9}}
\rnode{NA9.5}{}
\rnode{NA10}{\psframebox[linewidth=.5mm,fillstyle=solid,fillcolor=violet] {10}}
\rnode{NA11}{\psframebox[linewidth=.5mm,fillstyle=solid,fillcolor=cyan] {11}}
\rnode{NA12}{\psframebox*[linewidth=.5mm,fillstyle=solid,fillcolor=olive] {12}}
\rnode{NA13}{\psframebox[linewidth=.5mm,fillstyle=solid,fillcolor=red] {13}}
\rnode{NA14}{\psframebox[linewidth=.5mm,fillstyle=solid,fillcolor=red] {14}}
\rnode{NA15}{\psframebox[linewidth=.5mm,fillstyle=solid,fillcolor=violet] {15}}
\rnode{NA16}{\psframebox[linewidth=.5mm,fillstyle=solid,fillcolor=violet] {16}}
\rnode{NA17}{\psframebox[linewidth=.5mm,fillstyle=solid,fillcolor=olive] {17}}
\rnode{NA18}{\psframebox*[linewidth=.5mm,fillstyle=solid,fillcolor=violet] {18}}

\nput*[labelsep=4mm]{-90}{NA9.5}{\fontsize{8pt}{8pt}$S=[2,7]\cup[10,11]\cup[13,17]\quad|S|=13\quad2[S|\mathcal{P}]=|S|+3-2=14$}

\ncbar[angleA=90,arm=1mm,nodesep=2pt]{-}{NA4}{NA6}
\ncbar[angleA=90,arm=1mm,nodesep=2pt]{-}{NA13}{NA14}
\ncbar[angleA=90,arm=1mm,nodesep=2pt]{-}{NA9}{NA11}
\ncbar[angleA=90,arm=1mm,nodesep=2pt]{-}{NA15}{NA16}

\ncbar[angleA=90,arm=2mm,nodesep=2pt]{-}{NA3}{NA5}
\ncbar[angleA=90,arm=2mm,nodesep=2pt]{-}{NA12}{NA17}

\ncbar[angleA=90,arm=3mm,nodesep=2pt]{-}{NA2}{NA8}
\ncbar[angleA=90,arm=3mm,nodesep=2pt]{-}{NA10}{NA18}

\ncbar[angleA=90,arm=4mm,nodesep=2pt]{-}{NA1}{NA7}
\end{pspicture}

\vspace{.75cm}
\caption{}\label{asdfjkl;figure2}
\vspace{1mm}
\end{center}
\end{figure}

Figure \ref{asdfjkl;figure2} provides a few 
diagrammatic examples 
of the calculation of the number $2[S|\mathcal{P}]$. 
In each example, the coloring scheme 
indicates the partition defined by any word 
$(i_1,\ldots, i_{2k})\in \{1,\ldots, l\}^{2k}$ with 
exactly six values (so $l\geq 6$ necessarily), the level 
sets of which are indicated by blocks with common colors. 
The bracketings in each example define the pair partition 
\[
\mathcal{P}=\{\{1,7\},\{2,8\},\{3,5\},\{4,6\},\{9,11\},\{10,18\},\{12,17\},\{13,14\},\{15,16\}\}.
\]
Sets $S\subset [1,18]$ are indicated by bold borders 
drawn around an element, an element is a member of $S$ 
in each case if its border is emboldened. Evidently, 
the diagrams in Figure \ref{asdfjkl;figure2} will 
respectively contribute the rational progressions 
$\frac{1}{8}-\frac{1}{16}\mathbf{N}_{\geq 0},
-2-\frac{1}{4}\mathbf{N}_{\geq 0},
-\frac{5}{6}-\frac{1}{6}\mathbf{N}_{\geq 0},
\frac{3}{8}-\frac{1}{8}\mathbf{N}_{\geq 0}$,
and $\frac{1}{14}-\frac{1}{14}\mathbf{N}_{\geq 0}$
to the set of candidate poles for the meromorphic 
continuation of the integral $L(\mathcal{P};H)$.

The remaining sections will be devoted to the proof of Theorem \ref{asdfjkl;404}. 
Regarding the mere \emph{existence} of the meromorphic continuation, the result 
is not surprising and probably exists already in the 
literature in one form or another. 
Investigations of similar types of integrals have preoccupied many 
prominent researchers \cites{MR0256156,MR878239,MR0247457,MR0166596,MR1384760}. The pair
partition integrals 
$L(\{\{j_1^1,j_1^2\},\ldots,\{j_k^1,j_k^2\}\};H)$ bear a strong resemblance to the
standard beta integral 
$\int_0^1t^{\alpha-1}(1-t)^{\beta-1} dt=\Gamma(\alpha)\Gamma(\beta)/\Gamma(\alpha+\beta)$, 
along with its higher dimensional analog
\[
\int_{\Delta^n[0,1]}t_1^{\lambda_1-1}(t_2-t_1)^{\lambda_2-1}\cdots(1-t_n)^{\lambda_{n+1}-1}dt_1\cdots dt_n
=\frac{\Gamma(\lambda_1)\cdots\Gamma(\lambda_{n+1})}{\Gamma(\lambda_1+\cdots +\lambda_{n+1})}.
\]
Computation of these integrals is a straightforward matter, i.e. 
reduction of the multi-dimensional case to the one-dimensional 
case can be achieved by computing it as an iterated integral. Alternatively,
one can make a change of variable to express it as an integral over the standard simplex
$\Delta^n_1=\{(x_0,\ldots,x_n)\in\R^{n+1}:x_i>0, \sum x_i=1\}$, multiply by the 
convergence factor $e^{-\sum_i x_i}$ and use a homogeneity argument to obtain the 
desired equality.

Neither of these strategies works universally for the pair partition
integrals $L(\{\{j_1^1,j_1^2\},\ldots,\{j_k^1,j_k^2\}\};H)$ and the obstacle is 
easy to identify, the linear factors $s_{j_l^1}-s_{j_l^2}$ of 
the integrand are not necessarily differences of coordinates having 
adjacent indices so they do not vanish on the codimension one boundary strata of the 
simplex - instead, they vanish on lower dimensional boundary strata. To 
overcome this obstacle, we will generalize the problem and compute instead
the possible locations of the singularities in $\C^{2^{n}-1}$ of the meromorphic continuation
of the integral
\begin{equation}\label{asdfjkl;100}
\mathfrak{I}^n((\lambda_S)_{S\neq \varnothing})
=\int_{\Delta_n}\prod_{S\neq \varnothing}(x\cdot \bfone_S)^{\lambda_S} dx_1\wedge\ldots\wedge dx_n.
\end{equation}
Here the vectors $\bfone_S$ are the characteristic functions of the subsets of $\{1,\ldots, n\}$, so 
$x\cdot \bfone_S=\sum_{i\in S}x_i$, and $\Delta_{n}$ is the $dx_1\wedge\ldots\wedge dx_n$-oriented
simplex $\Delta_{n}=\{x\in\R^n: x_i>0, \sum x_i<1\}$. Analysis of the 
meromorphic continuation of (\ref{asdfjkl;100}) will be achieved by blowing up the 
simplex $\Delta_n$ in the following manner. In section 1 we define for $S\subset[1,n]$ 
the affine form $f_S(y)=-q(|S|)+y\cdot \bfone_S$ for $y\in \mathbf{R}^n$ and 
a strategically chosen function $q:\mathbf{N}_{\geq 0}\rightarrow \mathbf{N}_{\geq 0}$.
It is then proved that the map $F:\mathbf{R}^n\rightarrow \mathbf{R}^n$ defined by 
$F_i=\prod_{S\ni i} f_S$ is polynomial diffeomorphism from the region 
$\Omega=\bigcap_{S\neq\varnothing}\{f_S>0\}\subset \mathbf{R}^n$ onto 
the orthant $(0,\infty)^n\subset \mathbf{R}^n$ (Lemma \ref{asdfjkl;9}), which maps 
$\Omega'=\Omega\cap\{\sum_i F_i<1\}$ onto the simplex $\Delta_n$. 
For $\Re \lambda_S>0$, $\mathfrak{I}^n((\lambda_S)_{S\neq \varnothing})$ is therefore 
equal to the integral of the $n$-form pullback 
$F^\ast (\prod_{S\neq\varnothing}(x\cdot \bfone_S)^{\lambda_S}dx_1\wedge\ldots\wedge dx_n)$
over the region $\Omega'$. Evidently after rearranging terms, 
\begin{align*}
F^\ast &\pars{\prod_{S\neq\varnothing}(x\cdot \bfone_S)^{\lambda_S}dx_1\wedge\ldots\wedge dx_n} \\
&\hspace{2cm} =
\prod_{S\neq \varnothing}f_S^{\sum_{T\subset S}\lambda_T}P_S^{\lambda_S}\det(\partial F/\partial y)dy_1\wedge\ldots\wedge dy_n \\
&\hspace{2cm} =
\prod_{S\neq \varnothing}f_S^{|S|-1+\sum_{T\subset S}\lambda_T}P_S^{\lambda_S}
\frac{\det(\partial F/\partial y)}{\prod_{S\neq\varnothing}f_S^{|S|-1}}dy_1\wedge\ldots\wedge dy_n.
\end{align*}
where $P_S=\sum_{i\in S}\prod_{S\not\subset T\ni i}f_T$. 
However, in section 1 it is proved that the rational quotient
$R=\det(\partial F/\partial y)/\prod_{S\neq\varnothing}f_S^{|S|-1}$, written in (\ref{asdfjkl;700}), 
is a polynomial which is strictly positive on the closure of $\Omega'$
(Lemma \ref{asdfjkl;8} and related commentary). Furthermore, in 
section 2 it is proved that for each nonempty $S$,
$P_S$ is strictly positive 
on the closure of $\Omega'$. Therefore,
\begin{equation}\label{asdfjkl;6000}
\mathfrak{I}^n((\lambda_S)_{S\neq \varnothing})
=\int_{\Omega'}\prod_{S\neq \varnothing}f_S^{|S|-1+\sum_{T\subset S}\lambda_T} P_S^{\lambda_S} R dy_1\wedge\ldots\wedge dy_n.
\end{equation}
The integral will converge absolutely in an open subset of 
$\C^{2^n-1}$ (the product of half planes with positive real part, for 
instance). The meromorphic continuation arises from this expression 
by dissecting $\Omega'$ in such a way that the closure of every boundary component of the 
dissection intersects exactly one connected boundary stratum of $\Omega'$ 
of minimal dimension (for that component). The factor $f_S^{|S|-1+\sum_{T\subset S}\lambda_T}$ 
will contribute a progression of singularities along the 
hyperplanes defined by $|S|+\sum_{T\subset S}\lambda_T\in \N_{\leq 0}$.
This procedure is described in section 2.

The pair partition integrals $L(\mathcal{P};H)$ can 
easily be converted to special forms of the more general simplex integrals
$\mathfrak{I}^{2k}((\lambda_S)_{S\neq \varnothing})$ by making the change of variable $x_1=s_1$, and $x_i=s_i-s_{i-1}$
for $i\in\{2,\ldots, 2k\}$ which maps the increasing simplex $\Delta^{2k}[0,1]$ bijectively onto 
the solid simplex $\Delta_{2k}=\{x\in\R^{2k}:x_i>0,\sum x_i<1\}$. Therefore,
with $\mathcal{P}=\{\{j_1^1,j_1^2\},\ldots,\{j_k^1,j_k^2\}\}$,
\begin{align}
L(\mathcal{P};H)
&=\int_{\Delta^{2k}[0,1]}\prod_{l=1}^k|s_{j_l^1}-s_{j_l^2}|^{2H-2}ds_1\wedge\ldots \wedge ds_{2k} \notag\\
&=\int_{\Delta_{2k}}\prod_{l=1}^k\pars{\sum_{i=j^1_l\wedge j^2_l+1}^{j^1_l\vee j^2_l}x_i}^{2H-2}dx_1\wedge\ldots \wedge dx_{2k} \notag\\
&=\int_{\Delta_{2k}}\prod_{l=1}^k(x\cdot\bfone_{I(\{j_l^1,j_l^2\})})^{2H-2}dx_1\wedge\ldots \wedge dx_{2k} \label{asdfjkl;600}
\end{align}
which is precisely the integral $\mathfrak{I}^{2k}$ evaluated at the parameter $(\lambda_S)_{S\neq\varnothing}$
defined by 
\[
\lambda_S=
\begin{cases}
2H-2 & \text{if }S\in \mathcal{I}(\mathcal{P}) \\
0 & \text{otherwise}.
\end{cases}
\]
Thus,
\begin{equation}\label{asdfjkl;5000}
L(\mathcal{P}; H) =
\int_{\Omega'} \prod_{S\neq \varnothing}f_S^{|S|-1+2(H-1)[S|\mathcal{P}]} \prod_{S\in \mathcal{I}(\mathcal{P})}P_S^{2H-2} Rdy_1\wedge\ldots\wedge dy_n,
\end{equation}
and (\ref{asdfjkl;406}) can be rewritten as
\begin{align}
\mathbf{E}&\left[\int_{\Delta^{2k}[0,1]} dB^{i_1}_{t_1}\cdots dB^{i_{2k}}_{t_{2k}}\right] 
=\frac{H^k(2H-1)^k}{k!} \notag\\
&
\times\sum_{\mathcal{P}\leq (i_1,\ldots, i_{2k})}
\int_{\Omega'} \prod_{S\neq \varnothing}f_S^{|S|-1+2(H-1)[S|\mathcal{P}]} \prod_{S\in \mathcal{I}(\mathcal{P})}P_S^{2H-2} Rdy_1\wedge\ldots\wedge dy_n.\label{asdfjkl;4060}
\end{align}

For any $z\in \C$, 
$1-(|S|+z)/2[S|\mathcal{P}]$ is the value of $H$ such that 
the exponent $|S|-1+2(H-1)[S|\mathcal{P}]$ written above 
is equal to $-1-z$, so by requiring $-1-z\in \{-1,-2,\ldots\}$
one obtains the rational progressions specified 
in Theorem \ref{asdfjkl;404}. In this manner,
Theorem \ref{asdfjkl;404} and therefore also 
Corollaries \ref{asdfjkl;414},\ref{asdfjkl;1404} and \ref{asdfjkl;407} follow from the analysis 
of the singularities of $\mathfrak{I}^n$ given in 
section 2.

\section{Blowup of the orthant $\R^{n}_+$}
In this section we will give a constructive resolution 
of singularities for the orthant $\R^{n}_+=(0,\infty)^{n}$. To begin, 
for each subset $S\subset \{1,\ldots, n\}$, define 
the vector $\bfone_S$ to be the characteristic function 
of the set $S$ (i.e. if $n=5$ and $S=\{2,3,5\}$ then 
$\bfone_S=(0,1,1,0,1)$) and the affine form 
\beq{asdfjkl;0}
f_S(y_1,\ldots, y_n)
= -q(|S|)+ y\cdot\bfone_S=-q(|S|)+\sum_{i\in S} y_i
\eeq
where $q:\N_{\geq 0}\to \N_{\geq 0}$ is an arbitrary 
function that satisfies
\beq{asdfjkl;1}
q(0)=1
\quad\text{and}\quad
q(a)+q(b) < q(\max\{a,b\}+1)\quad\text{and}\quad
3q(a)\leq q(a+1)
\eeq
($q(r)=3^r$ seems to be a natural choice). A finite collection 
$\{S_1,\ldots, S_{r}\}\subset \mathfrak{P}(\{1,\ldots, n\})$ 
of distinct subsets will be called \emph{monotone} if for
any two of them one is fully contained in the other, or equivalently 
if there 
exists a permutation $\sigma\in \mathfrak{S}_{r}$ such that 
$S_{\sigma(1)}\subset \cdots\subset S_{\sigma(r)}$.
\begin{lemma}\label{asdfjkl;2}
If $S_0,S_1\subset\{1,\ldots, n\}$ and $y\in \R^{n}$ is any 
point such that
$f_{S_0}(y)=f_{S_1}(y)=0$ then $f_{S_0\cup S_1}(y)+y
\cdot\bfone_{S_0\cap S_1}$ is positive 
if $\{S_0,S_1\}$ is a monotone pair and is otherwise negative. 
In particular, if $f_{S_0}(y)=f_{S_1}(y)=0$, 
$y\cdot\bfone_{S_0\cap S_1} \geq 0$ 
and $\{S_0,S_1\}$ is not a monotone pair, then $f_{S_0\cup S_1}(y)$ 
is negative.
\end{lemma}
\begin{proof}
The hypothesis $f_{S_0}(y)=f_{S_1}(y)=0$ implies 
$y\cdot\bfone_{S_0}=q(|S_0|)$ and $y\cdot\bfone_{S_1}=q(|S_1|)$, 
thus
\begin{align*}
f_{S_0\cup S_1}(y)+y\cdot\bfone_{S_0\cap S_1}
&=-q(|S_0\cup S_1|)+y\cdot\bfone_{S_0\cup S_1}+y\cdot\bfone_{S_0\cap S_1} \\
&=-q(|S_0\cup S_1|)+y\cdot\bfone_{S_0}+y\cdot\bfone_{S_1} \\
&=-q(|S_0|+|S_1|-|S_0\cap S_1|)+q(|S_0|)+q(|S_1|).
\end{align*}
Set $M=\max\{|S_0|,|S_1|\}$ and $m=\min\{|S_0|,|S_1|\}$ so that 
\[
f_{S_0\cup S_1}(y)+y\cdot\bfone_{S_0\cap S_1}
=-q(M+m-|S_0\cap S_1|)+q(M)+q(m).
\]
If $S_0,S_1$ is a monotone pair then $m=|S_0\cap S_1|$ so 
\[
f_{S_0\cup S_1}(y)+y\cdot\bfone_{S_0\cap S_1}
=-q(M)+q(M)+q(m)=q(m)>0.
\]
If $S_0,S_1$ is not a monotone pair then $m-|S_0\cap S_1|\geq 1$ so 
\[
f_{S_0\cup S_1}(y)+y\cdot\bfone_{S_0\cap S_1}\leq -q(M+1)+q(M)+q(m)<0
\]
according to the assumed hypotheses (\ref{asdfjkl;1}) on $q$.
\end{proof}
Define the region $\Omega=\bigcap_{S\neq \varnothing}\{f_S>0\}\subset \R^{n}$ 
and the affine hyperplane $L_S=\{f_S=0\}\subset \R^{n}$ for each nonempty 
$S\subset\{1,\ldots, n\}$. 
\begin{lemma}\label{asdfjkl;3}
A set $\{f_{S_1},\ldots, f_{S_{r}}\}$ of affine forms of type 
(\ref{asdfjkl;0}) can vanish simunlaneously at a boundary point of 
the region $\Omega$ only if $\{S_1,\ldots, S_{r}\}$ is monotone 
and does not contain $\varnothing$, in which case the set 
\[
\Omega(\{S_1,\ldots,S_{r}\})=
\partial\Omega\bigcap (\cap_{S\in\{S_1,\ldots, S_{r}\}}L_{S})
\bigcap (\cup_{S\neq \varnothing,S\notin\{S_1,\ldots, S_{r}\}}L_{S})^c
\]
is a nonempty convex and relatively open subset of the codimension 
$r$ affine hyperplane $\cap_{S\in\{S_1,\ldots, S_{r}\}}L_{S}\subset \R^{n}$.
\end{lemma}
\begin{proof}
First, since $f_\varnothing(y)=-1$ for all $y$, if $y\in \partial \Omega$ 
and $f_{S_1}(y)=\ldots=f_{S_{r}}(y)=0$ then none of the $S_i$ can be empty. 
Next, choose distinct indices $i,j\in \{1,\ldots, r\}$. Since all boundary 
points of $\Omega$ must have strictly positive coordinates, evidently 
$\{S_i,S_j\}$ must be a monotone pair, for if not then 
$f_{S_i\cup S_j}(y)<0$ by Lemma \ref{asdfjkl;2}, but this impossible on 
$\partial\Omega$. Thus, we can conclude that either $S_i\subset S_j$ 
or $S_j\subset S_i$. Since this is true for any pair 
$i,j\in\{1,\ldots, r\}$ of distinct indices, it is clear that the set 
$\{S_1,\ldots, S_{r}\}$ must be monotone. Now the set 
$\Omega(\{S_1,\ldots,S_{r}\})$ can alternatively be written as
\begin{align*}
\Omega(&\{S_1,\ldots,S_{r}\}) \\
&=
(\cap_{S\in\{S_1,\ldots, S_{r}\}}\{y\cdot \bfone_S=q(|S|)\})
\bigcap (\cap_{S\neq \varnothing,S\notin\{S_1,\ldots, S_{r}\}}\{y\cdot \bfone_S>q(|S|)\})
\end{align*}
and each of the factors in the intersection
is convex, so the intersection must be as well. Also, it 
is clear that the intersection is relatively open in 
$\cap_{S\in\{S_1,\ldots, S_{r}\}}L_{S}$. It 
remains to prove that $\Omega(\{S_1,\ldots,S_{r}\})$ is nonempty. For this
we can extend the monotone set $\{S_1,\ldots, S_{r}\}$ into a monotone set 
$\{S_1,\ldots, S_{r},S_{r+1},\ldots, S_n\}$ of maximal length $n$. For 
$k\in \{1,\ldots n\}$ let $i_k\in\{1,\ldots ,n\}$ be the unique 
index such that $|S_{i_k}|=k$, let $j_k$ 
be the unique index such that $\{j_k\}=S_{i_k}\setminus S_{i_{k-1}}$
($S_{i_0}=\varnothing$), and let 
$\alpha\in\R^{n}$ be any vector such that 
\[
\begin{cases}
\alpha_k=q(k)\quad &i_k\in\{1,\ldots, r\} \\
q(k)<\alpha_k<\frac{q(k)+q(k+1)}{2} &i_k\in\{r+1,\ldots, n\}.
\end{cases}
\]
We claim that the vector $y$ defined by $y_{j_1}=\alpha_1$ and 
$y_{j_k}=\alpha_k-\alpha_{k-1}$ for $k\in\{2,\ldots n\}$ is an element 
of $\Omega(\{S_1,\ldots,S_{r}\})$. To prove this, observe 
that $y\cdot \bfone_{S_{i_k}}=\alpha_k$ so we only need to 
check that $y\cdot \bfone_S>q(|S|)$ for $S\notin\{S_1,\ldots, S_n\}$.
This is straightforward: let $k_S$ be the minimal number such 
that $S\subset S_{i_{k_S}}$, since $S\neq S_{i_{k_S}}$ evidently 
$|S|\leq k_S-1$ but also $S$ intersects $S_{i_{k_S}}\setminus S_{i_{k_S-1}}$
by the minimality of $k_S$ so 
\begin{align*}
y\cdot \bfone_S&\geq y\cdot \bfone_{S_{i_{k_S}}\setminus S_{i_{k_S-1}}} \\
&=\alpha_{k_S}-\alpha_{k_S-1}\\
&>q(k_S)-\frac{q(k_S-1)+q(k_S)}{2} \\
&=\frac{q(k_S)-q(k_S-1)}{2} \\
&\geq q(k_S-1) \\
&\geq q(|S|)
\end{align*}
according to the assumed hypotheses (\ref{asdfjkl;1}) on $q$ (this is 
where $3q(a)\leq q(a+1)$ is used).
\end{proof}

The following corollary follows immediately.
\begin{cor}\label{asdfjkl;5}
The boundary of $\Omega$ admits a decreasing filtration 
in closed sets given by 
\beq{asdfjkl;6}
\partial_r\Omega=
\bigcup_{r\leq \rho\leq n}\bigcup_{S_1\subset \cdots \subset S_{\rho}} \Omega(\{S_1,\ldots,S_{\rho}\})
\eeq
where the union is taken over all length not less than $r$ monotone 
lists of subsets which do not contain $\varnothing$. 
The top-dimensional stratum $\partial_r\Omega\setminus \partial_{r+1}\Omega$
in each filtration degree is relatively open in $\partial_r\Omega$ and 
its connected components are in bijection with the length-$r$ monotone 
lists of subsets which do not contain $\varnothing$.
\end{cor}
Next, define the map $F:\R^{n}\to \R^{n}$ by 
$F=(F_1,\ldots, F_n)$ with
\beq{asdfjkl;7}
F_i(y_1,\ldots, y_n)=\prod_{S\ni i} f_S(y_1,\ldots, y_n)
\eeq
\begin{lemma}\label{asdfjkl;8}
The jacobian determinant of $F$ is given by
\begin{align*}
\det &(\partial F/\partial y) \\
&=\prod_{S\neq \varnothing}f_S^{|S|-1}
\pars{\sum_{\varnothing \notin\{S_1,\ldots, S_n\}\subset\mathfrak{P}(\{1,\ldots, n\})} 
\det\pars{M^{\bfone_{S_1},\ldots,\bfone_{S_n}}}^2\prod_{\varnothing\neq P\notin\{S_1,\ldots, S_n\}}f_P}.
\end{align*}
\end{lemma}
Before the proving the lemma, a few remarks are in order.
First, $M^{\bfone_{S_1},\ldots,\bfone_{S_n}}$ denotes the matrix 
having the given ordered list of vectors as \emph{columns}. The matrix 
$M^{\bfone_{S_1},\ldots,\bfone_{S_n}}$ itself depends on the order 
of the list $S_1,\ldots, S_n$ but the determinant of its square does 
not, so the fact that the summands are indexed by \emph{subcollections} and 
not \emph{ordered subcollections} of $\mathfrak{P}(\{1,\ldots, n\})\setminus
\varnothing$ is not an issue. Most importantly,
the factor
\begin{equation}\label{asdfjkl;700}
R=
\sum_{\varnothing \notin\{S_1,\ldots, S_n\}\subset\mathfrak{P}(\{1,\ldots, n\})} 
\det\pars{M^{\bfone_{S_1},\ldots,\bfone_{S_n}}}^2\prod_{\varnothing\neq P\notin\{S_1,\ldots, S_n\}}f_P
\end{equation}
in $\det (\partial F/\partial y) $ must be strictly positive at all boundary points of 
$\Omega$, for it is a sum of functions which are all nonnegative on 
$\overline{\Omega}$ so if it vanishes at a boundary point then all summands 
must vanish at that point, but by Lemma \ref{asdfjkl;3} there exists a 
monotone list $\{S_1,\ldots, S_r\}$ of nonempty subsets of $\{1,\ldots, n\}$ 
such that $f_S$ vanishes at the point in question if and only if 
$S\in\{S_1,\ldots, S_r\}$ so we can enlarge the given list to a monotone list 
$\{S_1,\ldots, S_n\}$ and thus conclude that for this particular list the summand
\[
\det\pars{M^{\bfone_{S_1},\ldots,\bfone_{S_n}}}^2\prod_{\varnothing\neq P\notin\{S_1,\ldots, S_n\}}f_P,
\]
and therefore the entire sum, does not vanish at the specified boundary point. 
This being true of every point in $\partial \Omega$, we conclude that according 
to Lemma \ref{asdfjkl;8},
$\det (\partial F/\partial y)$ is equal to the product $\prod_{S\neq \varnothing}f_S^{|S|-1}$ times the 
polynomial $R$ defined in (\ref{asdfjkl;700}) which is strictly positive on the closure $\overline{\Omega}$.
\begin{proof}[Proof of Lemma \ref{asdfjkl;8}.]
Evidently
\[
\frac{\partial F_i}{\partial y_j} 
=\sum_{S\ni i}|\{j\}\cap S|\prod_{S'\ni i, S'\neq S}f_{S'} 
=\sum_{S\ni i,j}\prod_{S'\ni i, S'\neq S}f_{S'} 
=\sum_{S\ni i,j}\frac{1}{f_S}\prod_{S'\ni i}f_{S'}.
\]
Factoring out the $j$-independent product gives $\frac{\partial F_i}{\partial y_j} 
=\pars{\prod_{S'\ni i}f_{S'}}\pars{\sum_{S\ni i,j}f_S^{-1}}$
from which we conclude that the Jacobian matrix 
$\partial F/\partial y$ factors as $\partial F/\partial y =DA$ where 
$D$ is the diagonal matrix with $i$-th diagonal entry 
$\prod_{S'\ni i}f_{S'}$ and $A$ is the matrix with 
$\sum_{S\ni i,j} f_S^{-1}$ in row $i$ and column $j$. 
Since $D$ is diagonal, its determinant is easily computed by 
observing that the factor $f_S$ appears in precisely $|S|$ entries 
on the diagonal, and therefore $\det (\partial F/\partial y) = \prod_{S\neq \varnothing} f_S^{|S|}\det A$. 
The computation thus reduces to that of $\det A$. 
For this, we observe that the matrix $A$ is the sum
$A=\sum_S  f_S^{-1}\bfone_S^{\otimes 2}$ 
where $\bfone_S^{\otimes 2}$ denotes the matrix with $1$ in entry $i,j$ 
if and only if $i$ and $j$ are elements of $S$ (i.e. the matrix 
$\bfone_S^{\otimes 2}$, viewed as a function on the set $\{1,\ldots, n\}^2$, 
is literally the external tensor square of the vector $\bfone_S$, when 
the latter is viewed as 
a function on the set $\{1,\ldots, n\}$). In order to compute $\det A$, 
we can expand the sums in the exterior product which defines the determinant:
\[
Ae_1\wedge \ldots \wedge A e_n 
= \sum_{S_1,\ldots, S_n\subset \{1,\ldots, n\}} \pars{\prod_{i=1}^n\frac{1}{f_{S_i}} }
(\bfone_{S_1}^{\otimes 2}e_1\wedge \ldots \wedge \bfone_{S_n}^{\otimes 2}e_n )
\]
Here the sum is taken over \emph{all} lists 
of $n$ nonempty subsets of $\{1,\ldots, n\}$, 
with or without repetitions. Now $\bfone_{S_i}^{\otimes 2}e_i$ is $\bfone_{S_i}$ 
if $i\in S_i$ and is zero otherwise, thus $\bfone_{S_i}^{\otimes 2}e_i=(e_i\cdot\bfone_{S_i})\bfone_{S_i}$
and therefore
\begin{align*}
Ae_1&\wedge \ldots \wedge A e_n  \\
&= \sum_{S_1,\ldots, S_n\subset \{1,\ldots, n\}} \pars{\prod_{i=1}^n\frac{1}{f_{S_i}} }
((\bfone_{S_1}\cdot e_1)\bfone_{S_1}\wedge \ldots \wedge (\bfone_{S_n}\cdot e_n)\bfone_{S_n}) \\
&=\sum_{S_1,\ldots, S_n\subset \{1,\ldots, n\}} \pars{\prod_{i=1}^n\frac{e_i\cdot\bfone_{S_i}}{f_{S_i}} }
\det (M^{\bfone_{S_1},\ldots,\bfone_{S_n}})(e_1\wedge \ldots \wedge e_n)
\end{align*}
where $e_i=\bfone_{\{i\}}$ is the $i$-th standard basis vector. 
If $S_1,\ldots, S_{2^n-1}$ is an 
enumeration of the nonempty subsets of $\{1,\ldots, n\}$ then evidently
\begin{align*}
\det A &=
\sum_{S_1,\ldots, S_n\subset \{1,\ldots, n\}}  \pars{\prod_{i=1}^n\frac{e_i\cdot\bfone_{S_i}}{f_{S_i}} }
\det (M^{\bfone_{S_1},\ldots,\bfone_{S_n}})  \\
&=\sum_{1\leq i_1<\cdots <i_n\leq 2^{n}-1}\sum_{\sigma\in\mathfrak{S}_{n}}
\frac{\det (M^{\bfone_{S_{i_{\sigma 1}}},\ldots,\bfone_{S_{i_{\sigma n}}}})\prod_{j=1}^n(e_j\cdot\bfone_{S_{i_{\sigma j}}})}{\prod_{j=1}^nf_{S_{i_{\sigma j}}}} \\
&=\sum_{1\leq i_1<\cdots <i_n\leq 2^{n}-1}
\frac{\det (M^{\bfone_{S_{i_1}},\ldots,\bfone_{S_{i_n}}})}{\prod_{j=1}^nf_{S_{i_j}}}
\sum_{\sigma\in\mathfrak{S}_{n}}
\sgn(\sigma)\prod_{j=1}^n(e_j\cdot\bfone_{S_{i_{\sigma j}}}) \\
&=\sum_{1\leq i_1<\cdots <i_n\leq 2^{n}-1}
\frac{1}{\prod_{j=1}^nf_{S_{i_j}}}
\det (M^{\bfone_{S_{i_1}},\ldots,\bfone_{S_{i_n}}})^2 \\
&=\sum_{\{S_1,\ldots,S_n\}\subset\mathfrak{P}(\{1,\ldots,n\})\setminus\varnothing}
\frac{\prod_{\varnothing\neq P\notin\{S_1,\ldots, S_n\}}f_P}{\prod_{S\neq\varnothing}f_{S}}
\det (M^{\bfone_{S_{1}},\ldots,\bfone_{S_{n}}})^2.
\end{align*}
Therefore,
\begin{align*}
\det &(\partial F/\partial y) \\
&= \prod_{S\neq \varnothing} f_S^{|S|}
\pars{\sum_{\{S_1,\ldots,S_n\}\subset\mathfrak{P}(\{1,\ldots,n\})\setminus\varnothing}
\frac{\prod_{\varnothing\neq P\notin\{S_1,\ldots, S_n\}}f_P}{\prod_{S\neq\varnothing}f_{S}}
\det (M^{\bfone_{S_{1}},\ldots,\bfone_{S_{n}}})^2}
\end{align*}
and the lemma follows after factoring $\prod_{S\neq\varnothing} f_{S}^{-1}$
out of the sum.
\end{proof}

\begin{lemma}\label{asdfjkl;9}
The map $F$ is a polynomial diffeomorphism from the region $\Omega$ onto the orthant $\R^{n}_+$.
\end{lemma}
\begin{proof}
Injectivity is easily proved, since the partial derivatives 
$\partial F_i/\partial y_j$ are positive on $\Omega$. To prove surjectivity, choose any 
fixed element of $\Omega$ such as $(t,\ldots, t)$ with $t>q(n)/n$. If $(x_1,\ldots, x_n)\in\R^{n}_+$
is arbitrarily chosen and $v=(x_1,\ldots,x_n)-F(t,\ldots,t)$ then $(\partial F/\partial y)^{-1}v$ is a vector field on $\Omega$. Flowing 
through this vector field in one unit of time will carry $(t,\ldots, t)$ into a point 
$(y_1,\ldots, y_n)$ such that $F(y_1,\ldots, y_n)=(x_1,\ldots, x_n)$. Moreover, the resulting 
integral curve cannot leave $\Omega$, for if it did then its image would have to leave the  
orthant $\R^{n}_+$ but the image of this curve is the line segment connecting $F(t,\ldots, t)$
to $(x_1,\ldots, x_n)$ and the orthant is convex so this line segment cannot intersect 
the boundary of the orthant. 
This proves that the $F$-preimage $(y_1,\ldots, y_n)$ of $(x_1,\ldots, x_n)$
is indeed an element of $\Omega$. Since $(x_1,\ldots, x_n)$ was chosen arbitrarily, $F$ must 
map $\Omega$ surjectively onto $\R^{n}_+$.
\end{proof}

\section{Pullback of $\mathfrak{I}^n$ through $F$}

As described in the introduction, we are concerned 
with the singularities of the meromorphic continuation
to $\C^{2^{n}-1}$ of the integral
\[
\mathfrak{I}^{n}((\lambda_S)_{S\neq \varnothing})
=\int_{\Delta_n}\prod_{S\neq \varnothing}(x\cdot \bfone_S)^{\lambda_S} dx_1\wedge\ldots\wedge dx_n
\]
which converges and defines a holomorphic 
function in the region $\cap_{S\neq\varnothing}\{\operatorname{Re} \lambda_{S}>0\}$.
This integral can now be pulled back to $\Omega'=F^{-1}(\Delta_{n})=\Omega\cap\{\sum_i F_i<1\}$ 
through the map $F$ defined in the previous section, allowing a detailed 
study of the singularities:

\begin{align}
\mathfrak{I}^{n}((\lambda_S)_{S\neq \varnothing})
&=\int_{\Delta_{n}}\prod_{S\neq \varnothing}(x\cdot \bfone_S)^{\lambda_S}dx_1\wedge \ldots \wedge dx_n\notag\\
&=\int_{\Omega'}F^\ast\pars{\prod_{S\neq \varnothing}(x\cdot \bfone_S)^{\lambda_S}  dx_1\wedge\ldots\wedge dx_n}\notag \\
&=\int_{\Omega'}\prod_{S\neq \varnothing}(F\cdot \bfone_S)^{\lambda_S}(\det\partial F/\partial y)  dy_1\wedge\ldots\wedge dy_n.
\label{asdfjkl;102}
\end{align}
The integrand in (\ref{asdfjkl;102}) can be further 
simplified as follows,
\[
\prod_{S\neq \varnothing}(F\cdot\bfone_S)^{\lambda_S}
= \prod_{S\neq \varnothing}\pars{\prod_{T\supset S}f_T}^{\lambda_S}
\pars{\sum_{i\in S} \prod_{S\not\subset T\ni i}f_T}^{\lambda_S}  
= \prod_{S\neq \varnothing}f_S^{\sum_{T\subset S}\lambda_T}
P_S^{\lambda_S}
\]
where $P_S=\sum_{i\in S} \prod_{S\not\subset T\ni i}f_T$
as in the introduction.

Regarding $P_S$, we observe that it 
must be strictly positive on $\overline{\Omega'}$. The 
only points in question are the boundary points, and if the 
sum vanishes at a boundary point then \emph{all} of the summands
must vanish. This cannot happen, for as a result of the monotonicity 
property of the boundary (Corollary \ref{asdfjkl;5}), 
at any given boundary point $x\in\partial \Omega'$
there is a unique monotone list $T_1\subset \cdots \subset T_r$ 
such that 
$f_U(x)=0$ if and only if $U$ is one of the $T_j$. For any 
$S$, either $S\subset T_1$ in which case every summand in $P_S$ 
is positive, or from 
among these $T_j$ there is a maximal choice $T_\rho$
such that $T_\rho\subsetneq S$, i.e. 
$T_1\subset \cdots T_\rho\subsetneq S\subset T_{\rho+1}\subset \cdots\subset T_r$. 
In the latter case, as $T_\rho$ 
is a \emph{proper} subset of $S$ we are free to choose an element $i_x\in S$ 
such that 
$i_x\notin T_\rho\supset\cdots \supset T_1$. If $i_x\in T$ then either 
$f_T(x)>0$ or $T\supset S$. Consequently, the summand 
$\prod_{S\not\subset T\ni i_x}f_T$ is positive at $x$. This being 
true for any $x\in \partial \overline{\Omega'}$, evidently 
$P_S$ is positive on $\overline{\Omega'}$ for every $S$. 

From this we obtain the expression (\ref{asdfjkl;6000})
from the introduction:
\[
\mathfrak{I}^n((\lambda_S)_{S\neq \varnothing})
=\int_{\Omega'}\prod_{S\neq \varnothing}f_S^{|S|-1+\sum_{T\subset S}\lambda_T} P_S^{\lambda_S} R dy_1\wedge\ldots\wedge dy_n.
\]
To complete the proof of Theorem \ref{asdfjkl;404}, 
the boundary of $\Omega'$ can be covered by 
a partition of unity in such a way that 
in an open neighborhood of the support of every 
element of the partition
there is a monotone list $S_1\subset \cdots\subset S_r$
such that $df_{S_1}\wedge\ldots\wedge df_{S_r}$ is 
a nonzero $r$-form and such that every boundary 
hyperplane in the support of $\varphi$ is defined 
by one of the $f_{S_i}$. The list can then be completed 
to a full coordinate system by adding in other 
$f_{S_i}$ and possibly also $F_1+\ldots +F_n$ if the 
support of $\varphi$ intersects the boundary hypersurface 
defined by $F_1+\ldots +F_n=1$. This reduces the computation 
to that of a finite sum of integrals of the form 
\[
\int_{\Omega'}\prod_{S\neq \varnothing}f_S^{|S|-1+\sum_{T\subset S}\lambda_T} P_S^{\lambda_S} R \varphi dy_1\wedge\ldots\wedge dy_n
\]
where $\varphi$ is an element of the chosen 
partition of unity and the integrand can be rewritten 
using the coordinates $f_{S_1},\ldots,f_{S_n}$ in most cases
(or $f_{S_1},\ldots,f_{S_{n-1}},F_1+\ldots +F_n$ in 
other cases) as an integral 
of the form 
\[
\int_{0<f_{S_i}<1, 1\leq i\leq n} f_{S_1}^{|S_1|-1+\sum_{T\subset S_1}\lambda_T} \cdots f_{S_n}^{|S_n|-1+\sum_{T\subset S_n}\lambda_T} 
\Phi((\lambda_S)_{S\neq \varnothing})\varphi df_{S_1}\wedge\ldots\wedge df_{S_n}
\]
where $\Phi(f_{S_1},\ldots,f_{S_n},(\lambda_S)_{S\neq \varnothing})$ is an entire 
function of $(\lambda_S)_{S\neq \varnothing}\in \C^{2^n-1}$ taking values in the analytic and 
nonvanishing functions on the cube $0<f_{S_i}<1, 1\leq i\leq n$.

Now we can define a new set of variables $(\psi_S)_{S\neq \varnothing}\in \C^{2^n-1}$ 
and consider the integral
\[
\int_{0<f_{S_i}<1, 1\leq i\leq n} f_{S_1}^{|S_1|-1+\sum_{T\subset S_1}\lambda_T} \cdots f_{S_n}^{|S_n|-1+\sum_{T\subset S_n}\lambda_T} 
\Phi((\psi_S)_{S\neq \varnothing})\varphi df_{S_1}\wedge\ldots\wedge df_{S_n}
\]
which for $\Re\lambda_S>0$ is apparently a holomorphic 
function of the $2^{n+1}-2$ complex variables 
$(\lambda_S)_{S\neq\varnothing},(\psi_S)_{S\neq \varnothing}$, 
For $(\psi_S)_{S\neq \varnothing}\in \C^{2^n-1}$ fixed, 
this integral evidently extends to a meromorphic function on 
$\C^{2^n-1}$ in the variables $(\lambda_S)_{S\neq \varnothing}$
with poles in the union of hyperplanes defined by 
$|S_i|-1+\sum_{T\subset S_i}\lambda_T\in \{-1,-2,\ldots\}$ 
for $1\leq i\leq n$ as can be shown by expanding the factor 
$\Phi((\psi_S)_{S\neq \varnothing})\varphi $ into its taylor 
series in the coordinates $f_{S_1},\ldots, f_{S_n}$ and 
integrating termwise. 

Thus, the given integral extends to a meromorphic function 
on $\C^{2^{n+1}-2}$ with poles in the given list of hyperplanes
and the original integral is the restriction of this meromorphic 
function to the diagonal subspace $\{\lambda_S-\psi_S=0\}_{S\neq \varnothing}$. 
This completes the proof of Theorem \ref{asdfjkl;404}.


\begin{bibdiv}
\begin{biblist}


\bib{MR0256156}{article}{
   author={Atiyah, M. F.},
   title={Resolution of singularities and division of distributions},
   journal={Comm. Pure Appl. Math.},
   volume={23},
   date={1970},
   pages={145--150},
   issn={0010-3640},
   review={\MR{0256156}},
}


\bib{MR878239}{article}{
   author={Barlet, Daniel},
   title={Monodromie et p\^oles du prolongement m\'eromorphe de $\int_X\vert f\vert ^{2\lambda}\square$},
   language={French, with English summary},
   journal={Bull. Soc. Math. France},
   volume={114},
   date={1986},
   number={3},
   pages={247--269},
   issn={0037-9484},
   review={\MR{878239}},
}

\bib{MR2320949}{article}{
   author={Baudoin, Fabrice},
   author={Coutin, Laure},
   title={Operators associated with a stochastic differential equation
   driven by fractional Brownian motions},
   journal={Stochastic Process. Appl.},
   volume={117},
   date={2007},
   number={5},
   pages={550--574},
   issn={0304-4149},
   review={\MR{2320949 (2008c:60050)}},
   doi={10.1016/j.spa.2006.09.004},
}

\bib{MR0247457}{article}{
   author={Bern{\v{s}}te{\u\i}n, I. N.},
   author={Gel{\cprime}fand, S. I.},
   title={Meromorphy of the function $P^{\lambda }$},
   language={Russian},
   journal={Funkcional. Anal. i Prilo\v zen.},
   volume={3},
   date={1969},
   number={1},
   pages={84--85},
   issn={0374-1990},
   review={\MR{0247457}},
}




\bib{MR2604669}{book}{
   author={Friz, Peter K.},
   author={Victoir, Nicolas B.},
   title={Multidimensional stochastic processes as rough paths},
   series={Cambridge Studies in Advanced Mathematics},
   volume={120},
   note={Theory and applications},
   publisher={Cambridge University Press},
   place={Cambridge},
   date={2010},
   pages={xiv+656},
   isbn={978-0-521-87607-0},
   review={\MR{2604669 (2012e:60001)}},
}

\bib{MR0166596}{book}{
   author={Gel'fand, I. M.},
   author={Shilov, G. E.},
   title={Generalized functions. Vol. I: Properties and operations},
   series={Translated by Eugene Saletan},
   publisher={Academic Press, New York-London},
   date={1964},
   pages={xviii+423},
   review={\MR{0166596}},
}






\bib{MR1555421}{article}{
   author={Young, L. C.},
   title={An inequality of the H\"older type, connected with Stieltjes
   integration},
   journal={Acta Math.},
   volume={67},
   date={1936},
   number={1},
   pages={251--282},
   issn={0001-5962},
   review={\MR{1555421}},
   doi={10.1007/BF02401743},
}

\bib{MR1384760}{book}{
   author={Varchenko, A.},
   title={Multidimensional hypergeometric functions and representation
   theory of Lie algebras and quantum groups},
   series={Advanced Series in Mathematical Physics},
   volume={21},
   publisher={World Scientific Publishing Co., Inc., River Edge, NJ},
   date={1995},
   pages={x+371},
   isbn={981-02-1880-X},
   review={\MR{1384760}},
   doi={10.1142/2467},
}

\end{biblist}
\end{bibdiv}
\end{document}